\documentclass[a4paper,12pt,leqno]{amsart}
\usepackage[active]{srcltx}
\usepackage{verbatim}
\usepackage{epsfig,graphicx,color,mathrsfs}
\usepackage{graphicx}
\usepackage{amsmath,amssymb,amsthm,amsfonts}
\usepackage{amssymb}
\usepackage[english]{babel}
\usepackage[scrtime]{prelim2e}

\usepackage[left=2.7cm,right=2.7cm,top=3cm,bottom=3cm]{geometry}

\newtheorem{thm}{Theorem}[section]
\newtheorem{cor}[thm]{Corollary}
\newtheorem{lem}[thm]{Lemma}

\newtheorem{definition}[thm]{Definition}

\newcommand{\R}{{\mathbb{R}}}

\numberwithin{equation}{section}

\begin{document}
\title[Regularity and symmetry results]{Regularity and symmetry results for nonlinear degenerate elliptic equations}
\author{Francesco Esposito*, Berardino Sciunzi* and Alessandro Trombetta*}
\address{* Dipartimento di Matematica, UNICAL, Ponte Pietro Bucci 31B, 87036
Arcavacata di Rende, Cosenza, Italy.}

\email{esposito@mat.unical.it}

\email{sciunzi@mat.unical.it}
	
\email{alessandro.trombetta@unical.it}

\thanks{F. Esposito and B. Sciunzi were partially supported by PRIN project  2017JPCAPN (Italy): {\em Qualitative and quantitative aspects of nonlinear PDEs} and also by Gruppo Nazionale per l'Analisi Matematica, la Probabilit\`a e le loro Applicazioni (GNAMPA) of the Istituto Nazionale di Alta Matematica (INdAM)}

\begin{abstract}
In this paper we prove regularity results for a class of nonlinear degenerate elliptic equations of the form $\displaystyle -\operatorname{div}(A(|\nabla u|)\nabla u)+B\left( |\nabla u|\right) =f(u);$ in particular, we investigate the second order regularity of the solutions. As a consequence of these results, we obtain symmetry and monotonicity properties of positive solutions for this class of degenerate problems in convex symmetric domains via a suitable adaption of the celebrated moving plane method of Alexandrov-Serrin.
\end{abstract}

\maketitle

\section{Introduction}
In the first part of the paper we deal with regularity of weak solutions of the following nonlinear degenerate elliptic equation in divergence form
\begin{equation}  \label{equ1}
-\operatorname{div}(A(|\nabla u|)\nabla u)+B\left( |\nabla u|\right) =f(u) \qquad \text{in } \Omega,
\end{equation}
where $\Omega$ is any open set of $\mathbb{R}^{N}$, $N\geq 2$. 

The real function $A:\R^+ \rightarrow \R^+$ is of class $\mathcal{C}^1(\R^+)$ and it satisfies the following  assumptions:
\begin{eqnarray}
\label{hpA1}
&&\displaystyle \limsup_{t \rightarrow 0^+} tA(t)< +\infty;\\ \label{hpA2}
&&-1<\underset{t>0}{\inf }\,\, \displaystyle \frac{tA'(t)}{A(t)}=:m_{A} \leq  M_{A}:=\underset{t>0}{\sup }\,\,\frac{tA'(t)}{A(t)}<+\infty;\\ \label{hpA3}
&&\exists \; \tilde{\vartheta} \geq 0 \; \text{such that } A(t)\geq K \cdot t^{\tilde{\vartheta}} \; \text{for some } K>0, \; \forall t\in\R^+.
\end{eqnarray}
Assumption \eqref{hpA2} ensures that the differential operator in \eqref{equ1} satisfies ellipticity and monotonicity conditions, not necessarily of power type \cite{CM1, CM2,CM3}.

 Moreover we assume that:
\begin{equation} \label{hpA4}
\left(t \mapsto A(t) \text{ is increasing}\right) \quad \text{or} \quad
\begin{pmatrix}
A(t) \geq \tau(\mathcal{K}) > 0 \text{ on any compact set } \mathcal{K} \subset [0,+\infty)\\
\text{and } t \mapsto tA(t) \text{ is increasing}
\end{pmatrix}.
\end{equation}

We suppose that the real valued function $B: \overline{\R^+} \longrightarrow \overline{\R^+}$ is of class $\mathcal{C}^1(\overline{\R^+})$ and it satisfies
\begin{eqnarray}\label{hpB1}
&& B(0)=0; \\ \label{hpB3}
&& B'(t)\leq \hat{C} \cdot t A(t) \ \text{for some } \hat C >0, \; \forall t \in \R^+.
\end{eqnarray}

We remark that the assumptions on the nonlinearity $f :\R \rightarrow \R$ will be always specified. However the reader could think to $f$ as a locally Lipschitz continuous function.

Equation \eqref{equ1} has to be understood in the weak meaning, in particular we shall consider solutions of class $\mathcal{C}^{1,\alpha}$. This assumption is natural and, according to \cite{Di, Li, Te, T}, we give the following:

\begin{definition}\label{weaksol}
We say that $u\in C^{1,\alpha}(\Omega)$ is a weak solution to \eqref{equ1} if
\begin{equation}  \label{wequ1}
\int_\Omega A(|\nabla u|)\left(\nabla u,\nabla \psi \right) \, dx +\int_\Omega B\left( |\nabla u|\right) \psi \,dx=\int_\Omega f(u) \psi
\,dx
\end{equation}
for every $\psi \in C_{c}^{\infty }(\Omega )$.
\end{definition}

In the sequel we will frequently exploit the fact that the equation \eqref{equ1} is no
longer degenerate outside the critical set $\mathcal{Z}_u:=\{\nabla u=0\}$
of a solution $u$. Consequently it is also natural to assume that $u$ is of
class $\mathcal{C}^2$ outside $\mathcal{Z}_u$. We will use the notation $u_{i}:=u_{x_{i}}$, $i=1,\ldots, N$, to indicate the partial derivative of $u$ with respect to $x_i$. These are the classic derivatives since $u$ is of class $\mathcal{C}^1$. The
second derivatives of $u$ will be indicated with $u_{ij}$, $i,j=1,\ldots, N$. In this case, since $u$ is of class $\mathcal{C}^2$ only far from the singular set $\mathcal{Z}_u$, we agree that $u_{ij}$ coincides with the second derivatives of $u$ far from the singular set $\mathcal{Z}_u$, while $u_{ij}=0$ on the singular set $\mathcal{Z}_u$. At the beginning this is only a notation inspired by the Stampacchia's theorem, but with this definition $u_{ij}$ will represent actually the second distributional derivatives of $u$. We just remark that when $A(t)=t^{p-2}$ the operator reduces to the standard $p$-Laplacian and, in this case (when also $B \equiv 0$), by \cite{DS1, Sci1,Sci2}, it follows that $u \in W^{2,2}_{loc}(\Omega)$ if $1<p<3$, and that if $p\geq 3$ and the source term $f$ is strictly positive then $u \in W^{2,q}_{loc}(\Omega)$ for $q < \frac{p-1}{p-2}$. A generalization of local regularity results can be found in \cite{CDS,Riey}. All these arguments can be seen as an issue in the context of the Calder\'on-Zygmund theory for nonlinear degenerate problems (see \cite{Mingione1, Mingione2}).

However, the reader could think to the following model equation in the applications:
\begin{equation}\label{modelEquation}
	-\Delta_p u + \alpha |\nabla u|^p = f(u) \quad \text{in } \Omega,
\end{equation}
where $\alpha \geq 0$. In particular, in \eqref{modelEquation} we have chosen $A(t)=t^{p-2}$ and $B(t)=\alpha \cdot t^p$.
Now we state our local regularity results:
\begin{thm} \label{theoreg1} 
Let $u\in C^{1,\alpha }(\Omega )\cap C^{2}(\Omega \setminus
\mathcal{Z}_{u})$ be a weak solution to \eqref{equ1} with $f\in W^{1,\infty }(\Omega )$. Let $x_0 \in \Omega $ such that $B_{2\rho }(x_{0})\subset \Omega $ and assume that the function $A(\cdot)$ satisfies \eqref{hpA1}, \eqref{hpA2} and $B(\cdot)$ satisfies \eqref{hpB3}. Then
\begin{equation}  \label{reg1}
\int_{B_{\rho }(x_{0})}\frac{A(|\nabla u|)|\nabla u_{i}|^{2}}{|x-y|^{\gamma
}|u_{i}|^{\beta }}\,dx\leq \mathcal{C}\text{ \ \ \ }\forall i=1,...,N,
\end{equation}
uniformly for any $y \in B_\rho(x_0)$, with 
$$\mathcal{C}=\mathcal{C}(\gamma, m_A,M_A, \beta ,f,\|\nabla u\|_{\infty },\rho ,x_{0}),$$
for any $0\leq \beta <1$, $\gamma <N-2$ if $N\geq 3$, or $\gamma =0$ if $N=2$.
\end{thm}

We remark that Theorem \ref{theoreg1} holds without any sign assumption on the source term $f$. If we assume that $f$ has a sign, as a consequence of the previous result and we  obtain the summability properties of the inverse of the weight $A(|\nabla u|)$.

\begin{cor}
	\label{theoreg2} Let $u\in C^{1,\alpha }(\Omega )\cap C^{2}(\Omega \setminus
	\mathcal{Z}_{u})$ be a weak solution to \eqref{equ1} with $f\in W^{1,\infty }(\Omega )$
	and $f(u(x))=f(x)\geq c(\rho ,x_{0})>0$, in $B_{2\rho }(x_{0})\subset \Omega $ for
	some $\rho =\rho (x_{0})>0$. Let us assume that $A(\cdot)$ satisfies \eqref{hpA1}, \eqref{hpA2}, \eqref{hpA3} and that $B(\cdot)$ satisfies \eqref{hpB3}. Then
	\begin{equation}  \label{reg2}
	\int_{B_{\rho }(x_{0})}\frac{1}{(A(|\nabla u|))^{\sigma}}\frac{1}{|x-y|^\gamma}dx \leq \mathcal{C},
	\end{equation}
	with $1< \sigma <1+\frac{1}{\tilde\vartheta}$, $\gamma <N-2$ if $N\geq 3$, $%
	\gamma =0$ if $N=2$ and 
	\begin{equation*}
	\mathcal{C}=\mathcal{C}(\gamma ,m_{A},M_{A},f,\|\nabla u\|_{\infty },\rho ,x_{0},\sigma).
	\end{equation*}
	The same result holds true if we assume that $f(u(x))=f(x)\leq c(\rho, x_0)<0$ in $B_{2\rho}(x_0) \subset \Omega$.	Moreover $\left\vert \{A\left( \left\vert \nabla u\right\vert \right)=0\}\right\vert =0$ \emph{(\emph{in particular} $|\mathcal{Z}_u|=0$)}.
\end{cor}

The lack of regularity of the solutions of \eqref{equ1} is one of the greatest difficulty in the applications. In general, if we consider solutions of \eqref{equ1} in a general bounded smooth domain then the critical set $\mathcal{Z}_u$ may be very irregular and estimates of this kind are not available. Theorem \ref{theoreg1}  is actually an estimate on the way the operator degenerate near the critical set. We start observing that the estimates in Theorem \ref{theoreg1} and in Corollary \ref{theoreg2}, holds in a general compact set of $\Omega$. The same regularity holds all over the domain once we assume that there are no critical points of the solutions up to the boundary, namely $\mathcal{Z}_u\cap\partial\Omega=\emptyset$. This is an abstract assumption always verified each time we may exploit the Hopf's boundary lemma, see \cite{PSB}. The global regularity results  follow via a covering argument that can be found in \cite{CDS}.\\

However we will show that we can efficiently work in the weighted Sobolev space $W^{1,2}_\rho(\Omega)$ using only the estimates proved in Theorem \ref{theoreg1}. In particular, we will prove that if $f(s)>0$ for $s>0$ and $u$ is a solution of \eqref{equ1}, considering the weight $\rho=A(|\nabla u|)$, then it holds the following weighted
Poincar\'e's type inequality:
\begin{thm}\label{WPI}
	Let $\Omega' \subseteq \Omega $ and $u\in W_{0,\rho}^{1,2}\left( \Omega ^{\prime }\right) $. Then%
	\begin{equation}  \label{wpi}
	\int_{\Omega ^{\prime }}u^{2}dx\leq C_{P}\left( \Omega ^{\prime }\right)
	\int_{\Omega ^{\prime }}A(|\nabla u|)\left\vert \nabla u\right\vert
	^{2}dx,
	\end{equation}%
	where $C_{P}\left(\Omega ^{\prime }\right) \rightarrow 0$ when $\left\vert \Omega^{\prime}\right\vert \rightarrow 0$.
\end{thm}

Thanks to the weighted Poincar\'e's type inequality obtained in Theorem \ref{WPI}, we obtain the following weak comparison principle in small domains:

\begin{thm}\label{WCP}
	Let $u\in C^{1,\alpha }(\Omega )\cap C^{2}(\Omega \setminus
	\mathcal{Z}_{u})$ and $v\in C^{1,\alpha }(\Omega)\cap C^{2}(\Omega \setminus \mathcal{Z}_{v})$
	be weak solutions to \eqref{equ1}. Let us assume that $A(\cdot)$ satisfies \eqref{hpA1}, \eqref{hpA2}, \eqref{hpA3}, \eqref{hpA4}, that $B(\cdot)$ satisfies \eqref{hpB1}, \eqref{hpB3} and that $f$ is a locally Lipschitz continuous function with $f(s)>0$ for $s>0$. Let $\Omega'\subseteq \Omega $
	be open and suppose $u\leq v$ on $\partial \Omega'$, then there
	exists $\delta >0$ such that, if $\left\vert \Omega'\right\vert
	\leq \delta $, then $u\leq v$ in $\Omega'$.
\end{thm}

\

The second part of the paper is devoted to the study of the qualitative properties of solutions to the following quasilinear degenerate elliptic problem:
\begin{equation}\label{equ2}
\begin{cases}
-\operatorname{div}(A(|\nabla u|)\nabla u)+B\left( |\nabla u|\right) =f(u) \qquad &\text{in } \Omega\\
u>0 & \text{in } \Omega\\
u=0 &\text{on } \partial \Omega,
\end{cases}
\end{equation}
where $\Omega$ is a bounded smooth domain of $\R^N$, $N \geq 2$, and $f: \R \rightarrow \R$ is a positive locally Lispchitz continuous function. The functions $A(\cdot)$ and $B(\cdot)$ satisfy the previous assumptions \eqref{hpA1}, \eqref{hpA2}, \eqref{hpA3}, \eqref{hpA4}, \eqref{hpB1} and \eqref{hpB3}.

The technique which is mostly used in this part of the paper is the well-known moving plane method which goes back to the seminal works of Alexandrov \cite{A} and Serrin \cite{serrin}. See also the celebrated papers of Berestycki-Nirenberg \cite{BN} and Gidas-Ni-Nirenberg \cite{GNN}. Such a technique can be performed in general domains providing partial monotonicity results near the boundary and symmetry when the domain is convex and
symmetric. In all these papers, the moving plane procedure is applied in the semilinear case, in particular working with the Laplacian. We refer the reader to several papers where it is possible to find a generalization of the moving plane method for equations involving the $p$-Laplace operator in bounded domains \cite{Da2, DP, DamSciunzi, EMS, ES, Sci1, Sci2, SZ}. However, the reader could think to the following modelling problem in the applications:
\begin{equation}\label{model}
	\begin{cases}
		-\Delta_p u + \alpha |\nabla u|^{p} =f(u) \qquad &\text{in } \Omega\\
		u>0 & \text{in } \Omega\\
		u=0 &\text{on } \partial \Omega,
	\end{cases}
\end{equation}
where $\alpha \geq 0$ and $f$ is a locally Lipschitz continuous function with $f(s)>0$ for $s>0$.

We now state our  result:

\begin{thm}\label{symmetrybdd}
	Let $\Omega$ be a bounded smooth
	domain of $\R^N$ which is strictly convex in the
	$x_1$-direction and symmetric with respect to the hyperplane $\{x_1
	= 0\}$. Let $u \in C^{1,\alpha}(\overline{\Omega})$ be a positive solution of problem
	\eqref{equ2} with  $f$ a locally Lipschitz continuous function such that $f(s)>0$ for $s>0$. Moreover, let us assume that $A(\cdot )$ satisfies \eqref{hpA1}, \eqref{hpA2}, \eqref{hpA3}, \eqref{hpA4} and $B(\cdot)$ satisfies \eqref{hpB1}, \eqref{hpB3}. Then it follows that $u$ is symmetric with respect to the hyperplane $\{x_1 = 0\}$ and increasing in the
	$x_1$-direction in $\Omega \cap \{x_1 < 0\}$.\\
	\noindent In particular if the domain is a ball, then the solution is radial and radially decreasing, i.e.
	$$\frac {\partial  u}{\partial  r}(r)<0.$$
\end{thm}

\noindent Our paper is structured as follows:

\begin{itemize}
	\item[-] In Section \ref{sec2} we prove the local regularity results given by Theorem \ref{theoreg1} and by Corollary \ref{theoreg2} for equation \eqref{equ1}. As a consequence of these results we obtain two essential tools given by the weak comparison principle to compare solutions of \eqref{equ1} and by the weighted Poincar\'e type inequality.
	
	\item[-] In Section \ref{sec3} we show Theorem \ref{symmetrybdd}. We prove here the key Lemma \ref{conncomp} borrowing some ideas contained in \cite{BEV, EMM, EMS, MMPS}. After that, we develop a nice variant of the well-known moving plane method of Alexandrov and Serrin (see \cite{A,serrin}).
\end{itemize}

\section{Second order regularity of the solutions} \label{sec2}
The aim of this section is to prove the second order regularity result of solutions to \eqref{equ1}. \\

We start this section by proving the first local regularity result:

\begin{proof}[Proof of Theorem \ref{theoreg1}]
We start observing that, if $u\in C^{1,\alpha }(\Omega )\cap C^{2}(\Omega
\setminus \mathcal{Z}_u)$ is a solution to \eqref{equ1}, then any derivative of the
solution $u$, is also a solution to the linearized equation
\begin{equation}  \label{linequ}
\begin{split}
L_{u}(u_{i},\psi )=&\int_\Omega A(|\nabla u|)(\nabla u_{i},\nabla \psi )\,dx+
\int_\Omega \frac{A^{\prime }\left( \left\vert \nabla u\right\vert \right) }{\left\vert \nabla u\right\vert }\left( \nabla u,\nabla u_{i}\right) \left( \nabla u,\nabla \psi \right) dx\\
&+\int_\Omega \frac{B^{\prime }\left( \left\vert \nabla u\right\vert \right) }{\left\vert \nabla u\right\vert }\left( \nabla u,\nabla u_{i}\right) \psi \,dx-\int_\Omega f_{i} \, \psi \,dx=0,
\end{split}
\end{equation}
for every $\psi \in C_{c}^{\infty }(\Omega \setminus \mathcal{Z}_u)$. This follows
just putting $\psi _{i}$ as test function in \eqref{wequ1} and integrating
by parts.

We will use a regularization argument. For every $\varepsilon >0$ we set 
$$ G_{\varepsilon }(t):= (2t-2\varepsilon )\chi _{\lbrack \varepsilon
,2\varepsilon ]}(t)+t\chi _{\lbrack 2\varepsilon ,\infty )}(t)  \quad \text{for } t>0$$
and $G_\varepsilon(t):=-G_{\varepsilon }(-t)$ for $t\leq 0$, where $\displaystyle \chi_{\lbrack a,b]}(\cdot )$ denotes the characteristic function of $[a,b]$. We will assume that the ball $B_{2\rho }(x_{0})$ is contained in $\Omega $ and
we consider a cut-off function $\varphi_\rho \in C_{c}^{\infty }(B_{2\rho }(x_{0}))$ such that%
\begin{equation}  \label{cut_fun}
\varphi_\rho =1\text{ in }B_{\rho }(x_{0})\text{ and }|\nabla \varphi_\rho |\leq \frac{%
2}{\rho }\,\text{.}
\end{equation}

\noindent For $\beta \in \lbrack 0,1)$ and $\gamma <N-2$ if $N\geq 3$, or $\gamma
=0 $ for $N=2$, we set%
\begin{equation}  \label{T_H}
T_{\varepsilon }(t):=\frac{G_{\varepsilon }(t)}{|t|^{\beta }}\ , \ \ H_{\delta}(t):= \frac{G_{\delta}(t)}{|t|^{\gamma+1}}\ .
\end{equation}
We consider the test function
\begin{equation}  \label{test_fun}
\psi =T_{\varepsilon }(u_{i})\ H_{\delta }(|x-y|)\ \varphi_\rho ^{2} \\
=T_{\varepsilon }(u_{i})\ H_{\delta }\ \varphi_\rho ^{2}\,.
\end{equation}

\noindent According to \eqref{T_H}, it follows that such a test function can be
plugged in the linearized equation \eqref{linequ}, since it vanishes in a
neighbourhood of the critical set $\mathcal{Z}_u$. Consequently we get 
\begin{equation}\label{lin_phi}
\begin{split}
\int_{\Omega}A(|\nabla u|)|\nabla u_{i}|^2 T'_{\varepsilon}(u_{i}) H_{\delta} \varphi_\rho^2 \,dx &+\int_{\Omega}\frac{A'(|\nabla u|)}{|\nabla u|}(\nabla u,\nabla u_{i})^2 T'_{\varepsilon}(u_{i}) H_{\delta} \varphi_\rho^2\,dx\\ 
&+\int_{\Omega}A(|\nabla u|)(\nabla u,\nabla_x H_\delta) T_\varepsilon(u_{i}) \varphi_\rho^2\,dx \\
&+\int_{\Omega}\frac{A'(|\nabla u|)}{|\nabla u|}(\nabla u,\nabla u) (\nabla u,\nabla_x H_\delta) T_\varepsilon (u_{i}) \varphi_\rho^2\,dx\\
&+2\int_{\Omega}A(|\nabla u|)(\nabla u,\nabla \varphi_\rho) T_\varepsilon(u_{i}) H_\delta  \varphi_\rho \,dx\\
&+ 2\int_{\Omega}\frac{A'(|\nabla u|)}{|\nabla u|}(\nabla u,\nabla u) (\nabla u,\nabla \varphi_\rho) T_{\varepsilon }(u_{i})\ H_{\delta } \varphi_\rho \,dx\\
&+\int_{\Omega}\frac{B'\left(|\nabla u|\right)}{\left\vert \nabla u\right\vert}(\nabla u,\nabla u_{i})T_{\varepsilon}(u_{i})H_{\delta}\varphi_\rho ^{2}\,dx\\
&=\int_\Omega f_{i} \cdot T_{\varepsilon}(u_{i})H_{\delta}\varphi_\rho ^{2}\,dx\,.
\end{split}
\end{equation}
Now we set
\begin{equation}  \label{intI}
\begin{split}
I_{1}(\varepsilon ,\delta)=&\int_{\Omega}A(|\nabla u|)|\nabla u|^2 T'_{\varepsilon}(u_{i}) H_{\delta} \varphi_\rho^2\,dx \\
I_{2}(\varepsilon ,\delta)=&\int_{\Omega}\frac{A'(|\nabla u|)}{%
|\nabla u|}(\nabla u,\nabla u_{i})^{2}\ T_{\varepsilon }^{\prime }(u_{i})\
H_{\delta }\ \varphi_\rho ^{2}\,dx\\
I_{3}(\varepsilon ,\delta )=&\int_\Omega A(|\nabla u|)(\nabla u_{i},\nabla
_{x}H_{\delta })\ T_{\varepsilon }(u_{i})\ \varphi_\rho ^{2}\,dx \\
I_{4}(\varepsilon ,\delta )=&\int_\Omega \frac{A^{\prime }(|\nabla u|)}{%
|\nabla u|}(\nabla u,\nabla u_{i})(\nabla u,\nabla _{x}H_{\delta })\
T_{\varepsilon }(u_{i})\ \varphi_\rho ^{2}\,dx \\
I_{5}(\varepsilon ,\delta )=&2\int_\Omega A(|\nabla u|)(\nabla u_{i},\nabla
\varphi_\rho )\ T_{\varepsilon }(u_{i})\ H_{\delta } \varphi_\rho \,dx \\
I_{6}(\varepsilon ,\delta )=&2\int_\Omega \frac{A^{\prime }(|\nabla u|)}{|\nabla u|}(\nabla u,\nabla u_{i})(\nabla u,\nabla \varphi_\rho ) T_{\varepsilon}(u_{i})\ H_{\delta }\ \varphi_\rho \,dx \\
I_{7}(\varepsilon ,\delta )=&\int_\Omega \frac{B^{\prime }\left( |\nabla u|\right) }{\left\vert \nabla u\right\vert }(\nabla u,\nabla
u_{i})T_{\varepsilon }(u_{i})H_{\delta }\varphi_\rho ^{2}\,dx \\
I_{8}(\varepsilon ,\delta)=&\int_\Omega f_{i}\cdot T_{\varepsilon
}(u_{i})H_{\delta }\varphi_\rho ^{2}\,dx\,.
\end{split}
\end{equation}

\noindent Regarding the terms $I_{1}$ and $I_{2}$ we note that%
\begin{equation}\label{ip1}
I_{1}+I_{2}\geq \int_\Omega A(|\nabla u|)|\nabla u_{i}|^{2}\
T_{\varepsilon }^{\prime }(u_{i})\ H_{\delta }\ \varphi_\rho ^{2}\,dx
\end{equation}
when $A^{\prime }(|\nabla u|)$ is nonnegative, while
\begin{equation}
\begin{split}
I_1+I_2\geq&\int_\Omega A(|\nabla u|)|\nabla u_{i}|^{2}\
T_{\varepsilon}^{\prime}(u_{i})\ H_{\delta }\ \varphi_\rho ^{2}\,dx \\
&-\int_\Omega |A^{\prime }(|\nabla u|)||\nabla u||\nabla u_{i}|^{2}\
T_{\varepsilon }^{\prime }(u_{i})\ H_{\delta }\ \varphi_\rho ^{2}\,dx
\end{split}
\end{equation}
when $A^{\prime }(|\nabla u|)$ is negative. Therefore, by \eqref{hpA2}, in this case we have that
\begin{equation}\label{ip2}
I_{1}+I_{2}\geq (1+m_{A})\,I_{1}\,.
\end{equation}
Hence, by \eqref{lin_phi}, \eqref{ip1} and \eqref{ip2}, we deduce
\begin{equation}  \label{stima_i}
\min\{1, 1+m_{A}\}\int_\Omega A(|\nabla u|)|\nabla u_{i}|^{2}\ T_{\varepsilon
}^{\prime }(u_{i})\ H_{\delta }\ \varphi_\rho ^{2}\,dx\leq |I_{3}|+...+|I_{8}|\,.
\end{equation}

Now we give an estimate of the right hand side of \eqref{stima_i}. In particular, using  the fact that $|T_{\varepsilon }(t)|\leq t^{1-\beta}$ and the weighted Young's inequality $ab\leq \vartheta a^{2}+\frac{b^{2}}{%
4\vartheta }$, we deduce that%
\begin{equation}  \label{I3_I4}
\begin{split}
&\limsup_{\delta \rightarrow 0}(|I_{3}|+|I_{4}|) \\
&\leq  (\gamma+1) \int_{\Omega}\left(A(|\nabla u|)+|A^{\prime}(|\nabla u|)||\nabla u|\right) |\nabla u_{i}|\frac{|T_{\varepsilon}(u_{i})|}{|x-y|^{\gamma +1}}\varphi_\rho ^{2}\,dx \\
&\underset{\text{by } \eqref{hpA2}}{\leq} (\gamma+1) (1+|M_{A}|)\int_\Omega A(|\nabla u|)|\nabla u_{i}|\frac{|T_{\varepsilon} (u_{i})|}{|x-y|^{\gamma +1}}\varphi_\rho^2\,dx \\
&\leq(\gamma+1) (1+|M_{A}|)\int_{\Omega}\frac{\sqrt{A(|\nabla u|)} |\nabla
u_{i}|}{|x-y|^{\frac{\gamma}{2}}|u_{i}|^{\frac{\beta}{2}}} \chi_{\{|u_{i}|\geq \varepsilon \}} \varphi_\rho  \frac{\sqrt{A(|\nabla u|)} |u_{i}|^{\frac{2-\beta }{2}}}{|x-y|^{\frac{\gamma +2}{2}}}\varphi_\rho \,dx \\
&\leq \vartheta \int_\Omega \frac{A(|\nabla u|)|\nabla u_{i}|^{2}}{%
|x-y|^{\gamma }|u_{i}|^{\beta }}\chi _{\{|u_{i}|\geq \varepsilon \}}\varphi_\rho
^{2}\,dx+\frac{(\gamma+1)^{2}(1+|M_{A}|)^{2}}{4\vartheta }\int_\Omega \frac{%
A(|\nabla u|)|u_{i}|^{2-\beta }}{|x-y|^{\gamma +2}}\varphi_\rho ^{2}\,dx \\
&\leq \vartheta \int_\Omega \frac{A(|\nabla u|)|\nabla u_{i}|^{2}}{%
|x-y|^{\gamma }|u_{i}|^{\beta }}\chi _{\{|u_{i}|\geq \varepsilon \}}\varphi_\rho
^{2}\,dx+\frac{(\gamma+1)^{2}(1+|M_{A}|)^{2}}{4\vartheta }\int_\Omega \frac{%
A(|\nabla u|)|\nabla u|^{2-\beta }}{|x-y|^{\gamma +2}}\varphi_\rho ^{2}\,dx \\
&\leq \vartheta \int_\Omega \frac{A(|\nabla u|)|\nabla u_{i}|^{2}}{%
|x-y|^{\gamma }|u_{i}|^{\beta }}\chi _{\{|u_{i}|\geq \varepsilon \}}\varphi_\rho
^{2}\,dx+C_{3,4}(\gamma, M_A, M, L, \vartheta),
\end{split}
\end{equation}
where, in the last line of \eqref{I3_I4}, we also used the fact that $tA(t)$ is locally bounded and we have set
\begin{equation}
\begin{split}
M&:=\max \left\{ \sup_{y\in \Omega }\int_{B_{2\rho }(x_{0})}\frac{1}{|x-y|^{\gamma}}dx; \ \sup_{y\in \Omega }\int_{B_{2\rho }(x_{0})}\frac{1}{|x-y|^{\gamma +2}}dx\right\} , \\
L&:=\sup_{x\in B_{2\rho }(x_{0})}|\nabla u|\,.
\end{split}
\end{equation}
Exploiting the fact that $|\nabla \varphi_\rho |\leq \frac{2}{\rho }$, $%
|T_{\varepsilon }(t)|\leq t^{1-\beta }$ and the weighted Young's inequality, we also
get that%
\begin{equation}  \label{I5_I6}
\begin{split}
\limsup_{\delta \rightarrow 0}(&|I_{5}|+|I_{6}|)\\
&\leq 2\int_{\Omega}\left[A(|\nabla u|)+|A'(|\nabla u|)||\nabla u| \right] |\nabla u_{i}||\nabla \varphi_\rho |\frac{|T_{\varepsilon }(u_{i})|}{|x-y|^{\gamma }}\varphi_\rho \,dx \\
&\underset{\text{by } \eqref{hpA2}}{\leq} \frac{4(1+|M_{A}|)}{\rho }\int_\Omega \frac{\sqrt{A(|\nabla u|)}\
|\nabla u_{i}|}{|x-y|^{\frac{\gamma }{2}}|u_{i}|^{\frac{\beta }{2}}}\chi
_{\{|u_{i}|\geq \varepsilon \}}\varphi_\rho \ \ \frac{\sqrt{A(|\nabla u|)}\
|u_{i}|^{\frac{2-\beta }{2}}}{|x-y|^{\frac{\gamma }{2}}}\,dx \\
&\leq \vartheta \int_\Omega \frac{A(|\nabla u|)|\nabla u_{i}|^{2}}{|x-y|^{\gamma}|u_{i}|^{\beta }}\chi _{\{|u_{i}|\geq \varepsilon \}}\varphi_\rho
^{2}\,dx+\frac{4(1+|M_{A}|)^{2}}{\vartheta \rho ^{2}}\int_{B_{2\rho }(x_{0})}%
\frac{A(|\nabla u|)|\nabla u|^{2-\beta }}{|x-y|^{\gamma }}\,dx \\
&\leq \vartheta \int_\Omega \frac{A(|\nabla u|)|\nabla u_{i}|^{2}}{%
|x-y|^{\gamma }|u_{i}|^{\beta }}\chi _{\{|u_{i}|\geq \varepsilon \}}\varphi_\rho
^{2}\,dx+C_{5,6}(M_A, M,L, \vartheta, \varrho)\,,
\end{split}
\end{equation}
where in the last line of \eqref{I5_I6} we used the fact that the function $tA(t)$ is locally bounded.

Now, exploiting \eqref{hpA3}, \eqref{hpB3} and applying the weighted Young's inequality we obtain
\begin{equation}\label{I7} 
\begin{split}
\limsup_{\delta \rightarrow 0}|I_{7}|&\leq \int_\Omega |B^{\prime}\left(\left\vert \nabla u\right\vert \right) ||\nabla u_{i}|\frac{|T_{\varepsilon}(u_{i})|}{|x-y|^{\gamma }}\varphi_\rho ^{2}\,dx  \\
&\underset{\text{by } \eqref{hpB3}}{\leq} \hat{C}(1+|M_A|)\int_\Omega \frac{\sqrt{A(|\nabla u|)}\ |\nabla u_{i}|}{|x-y|^{\frac{\gamma}{2}}|u_{i}|^{\frac{\beta }{2}}}\chi_{\{|u_{i}|\geq \varepsilon\}} \varphi_\rho \ \ \frac{|\nabla u|\sqrt{A(|\nabla u|)} \ |u_{i}|^{\frac{2-\beta }{2}}}{|x-y|^{\frac{\gamma}{2}}}\varphi_\rho \,dx \\
&\underset{\text{by } \eqref{hpA3}}{\leq} \vartheta \int_\Omega \frac{A(|\nabla u|)|\nabla u_{i}|^{2}}{|x-y|^{\gamma}|u_{i}|^\beta} \chi_{\{|u_{i}|\geq \varepsilon \}} \varphi_\rho^2\,dx\\
& \qquad \; \; + \frac{\hat{C}^2(1+|M_A|)^2}{4 K\vartheta }\int_\Omega \frac{|\nabla u|^2 A(|\nabla u|)|\nabla u|^{2-\beta}}{|x-y|^\gamma} \varphi_\rho ^{2}\,dx \\
&\underset{\text{by } \eqref{hpA1}}{\leq} \vartheta  \int_\Omega \frac{A(|\nabla u|)|\nabla u_{i}|^{2}}{|x-y|^{\gamma }|u_{i}|^\beta} \chi _{\{|u_{i}|\geq \varepsilon \}} \varphi_\rho^2\,dx + \frac{L^2\hat{C}^2(1+|M_A|)^2}{4K\vartheta} \int_\Omega \frac{|\nabla u|^{1-\beta}}{|x-y|^{\gamma}}\varphi_\rho^2\,dx \\
&\leq \vartheta  \int_\Omega \frac{A(|\nabla u|)|\nabla u_{i}|^{2}}{|x-y|^{\gamma} |u_{i}|^{\beta}} \chi _{\{|u_{i}|\geq \varepsilon \}}\varphi_\rho^{2} \,dx+C_7(M_A, M, L, K, \vartheta)\,,
\end{split}
\end{equation}
where in the last line of \eqref{I7} we used the fact that the function $tA(t)$ is locally bounded.

Finally, setting $\displaystyle F=\sup_{x\in B_{2\rho }(x_{0})}\sum_{i=1}^{N}|f_{i}(x)|$,
we get
\begin{equation}\label{I8} 
\limsup_{\delta \rightarrow 0}|I_{8}|\leq F\int_\Omega \frac{|u_{i}|^{1-\beta }}{|x-y|^{\gamma}} \varphi_\rho ^{2}\,dx\leq C_8(F,M,L)\,.
\end{equation}
Taking into account \eqref{stima_i}, letting $\delta \rightarrow 0$,
exploiting the above estimates \eqref{I3_I4}, \eqref{I5_I6}, \eqref{I7}, \eqref{I8} and evaluating $T_{\varepsilon }^{\prime }$,
\begin{equation}
\begin{split}
\min\{1,&1+m_{A}\}\int_{\Omega} \frac{A(|\nabla u|)|\nabla u_{i}|^{2}}{|x-y|^{\gamma}} \left(\frac{G'_{\varepsilon}(u_{i})}{|u_{i}|^{\beta}}-\beta 
\frac{G_{\varepsilon}(u_{i})}{|u_{i}|^{1+\beta}}\right) \varphi_\rho^2\\
&-3\vartheta \int_\Omega \frac{A(|\nabla u|)|\nabla u_{i}|^{2}}{|x-y|^\gamma |u_{i}|^\beta} \chi_{\{|u_{i}|\geq \varepsilon\}}\varphi_\rho^2 \leq C_{3,4} + C_{5,6}  + C_7 + C_8 \,.
\end{split}
\end{equation}
Now we fix $\vartheta $ sufficiently small such that%
\begin{equation}  \label{stima_G}
\min\{1,1+m_{A}\}(1-\beta) -3\vartheta >0
\end{equation}
so that, passing to the limit for $\varepsilon \rightarrow 0$ and applying Fatou's Lemma we obtain
\begin{equation}
\int_{B_{\rho }(x_{0})}\frac{A(|\nabla u|)|\nabla u_{i}|^{2}}{|x-y|^{\gamma
}|u_{i}|^{\beta }}\,dx\leq \int_{B_{2\rho }(x_{0})}\frac{A(|\nabla
u|)|\nabla u_{i}|^{2}}{|x-y|^{\gamma }|u_{i}|^{\beta }}\varphi_\rho ^{2}\,dx\leq \mathcal{C}
\end{equation}
where $\mathcal{C}=\mathcal{C}(\gamma,m_A,M_A, \beta, f,\|\nabla u\|_\infty,\rho ,x_{0})$.
\end{proof}

As a consequence of this result it is possible to show Corollary \ref{theoreg2}.

\begin{proof}[Proof of Corollary \ref{theoreg2}]
	The proof repeats verbatim the arguments exploited in \cite[Theorem 1.3]{CDS}.
	
\end{proof}

As remarked in the introduction we remark that the global regularity results, in particular Theorem \ref{theoreg1} and Corollary \ref{theoreg2} proved in the whole $\Omega$, follow via a covering argument that can be found in \cite{CDS}. However we will show that we can efficiently work in the weighted Sobolev space $W^{1,2}_\rho(\Omega)$ using only the estimates proved in Theorem \ref{theoreg1}. 

Recall that, if $\rho \in L^{1}(\Omega)$, the space $W^{1,2}_\rho(\Omega)$ is defined
as the completion of $C^{\infty }(\overline{\Omega})$ under the norm
\begin{equation}\label{hthInorI}
\| v\|_{W^{1,2}_\rho}= \| v\|_{L^2 (\Omega)}+\| \nabla v\|_{L^2
	(\Omega, \rho)}
\end{equation}
where
$$
\| \nabla v\|^2_{L^2 (\Omega, \rho)}=\int_{\Omega}\rho |\nabla v|^2 \,dx.
$$
We also recall that $W^{1,2}_{0, \rho}$  is defined
as the completion of  $C^{\infty}_c(\overline{\Omega})$ under the norm
\begin{equation}\label{hthInorII}
\| v\|_{W^{1,2}_{0,\rho}}= \| \nabla v\|_{L^2 (\Omega, \rho)}.
\end{equation}

\begin{thm}[Weighted Sobolev inequality, \cite{DamSciunzi}]\label{thm: Sobolev}
	Let $\rho$ be a weight function such that
	\begin{equation}\label{eq:weight}
	\int_{\Omega} \frac{1}{\rho^\sigma|x-y|^{\gamma}}\leq \mathcal{C},
	\end{equation}
	with $1< \sigma <1+\frac{1}{\tilde{\vartheta}}$, $\gamma <N-2$ if $N\geq 3$, $%
	\gamma =0$ if $N=2$. Assume, in the case $N\geq 3$, without no loss of generality that
	$$\gamma>N-2 \sigma,$$
	which implies $N\sigma-2N+2\sigma+\gamma>0$. Then,  for any $w\in H^{1,2}_{0,\rho}(\Omega)$, there exists a constant $C_{s_\rho}$ such that
	\begin{equation}\label{FTFTnddjncj}
	\|w\|_{L^q(\Omega)}\leq C_{s_\rho}\|\nabla w\|_{L^2(\Omega,\rho)},
	\end{equation}
	for any $1\leq q< 2^*(\sigma )$ where
	\begin{equation}\label{eq:2*}
	\frac{1}{2^*(\sigma)}=\frac 12 - \frac 1N + \frac 1\sigma \left( \frac 12 - \frac{\gamma}{2N}\right).
	\end{equation}
\end{thm}

In particular, we use the previous result with the weight $\rho=A(|\nabla u|)$. Now we are ready to prove the weighted Poincar\'e type inequality.

\begin{proof}[Proof of Theorem \ref{WPI}] Choose $2< q< 2^*(\sigma)$. By H\"older inequality we get:
	\begin{equation}\label{poinstettema}
	\int_{\Omega}w^2\leq\left(\int_{\Omega}w^q\right)^\frac 2q |\Omega|^\frac{q-2}{q},
	\end{equation}
	and then  using Theorem \ref{thm: Sobolev} one has
	$$\int_{\Omega}w^2\leq {C_p}(\Omega) C_{s_\rho}^2 \int_{\Omega}\rho|\nabla w|^2.$$
	By \eqref{poinstettema} and direct computation it follows \eqref{wpi}.
	
\end{proof}

Now we recall some useful inequalities that hold for operators that satisfies ellipticity conditions (for the proof we refer to \cite{Da2, SZ}):  $ \forall \eta, \eta' \in \R^N$ there exist $C_1, C_2>0$ depending on $A$ such that
\begin{equation}\label{supportinequalities}
\begin{split}
\left( A(|\eta|)\eta -A(|\eta'|)\eta'\right) \left(\eta -\eta'\right) &\geq C_1 A(|\eta|+|\eta'|)|\eta	-\eta'|^2;\\
\left( A(|\eta|)\eta -A(|\eta'|)\eta'\right) \left(\eta -\eta'\right) &\leq C_2 A(|\eta|+|\eta'|)|\eta	-\eta'|.
\end{split}
\end{equation}

\begin{proof}[Proof of Theorem \ref{WCP}]
Let us consider the function 
$$w=\left( u-v\right)^+ \quad \text{in} \; \Omega'.$$
We observe that this function is bounded, $w=0$ on $\partial \Omega ^{\prime }$ and $w$ belongs to $W_{0,\rho}^{1,2}(\Omega)$; hence, it can be used as test function in \eqref{wequ1}. We have
\begin{eqnarray}\\\nonumber
\int_{\Omega'}\left( A(|\nabla u|)\nabla u-A(|\nabla v|)\nabla v,\nabla w\right) dx+\int_{\Omega'} \left( B\left( \left\vert \nabla u\right\vert \right) - B\left( \left\vert
\nabla v\right\vert \right) \right) w \, dx \\\nonumber
=\int_{\Omega'}\frac{f\left( u\right) -f\left( v\right) }{u-v} w^2 \, dx.
\end{eqnarray}
Using \eqref{supportinequalities} and the fact that $f$ is a locally Lipschitz continuous function, we obtain
\begin{equation}\label{i}
C_1\int_{\Omega'}A(\left\vert \nabla u\right\vert
+\left\vert \nabla v\right\vert) \vert \nabla w \vert^2 \, dx \leq \int_{\Omega'} \left| B\left(\left\vert \nabla u\right\vert \right) - B\left( \left\vert \nabla v\right\vert \right) \right| w \, dx + L_f\int_{\Omega'} w^2 \,dx,
\end{equation}
where $L_f>0$ is the Lipschitz constant. Moreover, using our assumptions \eqref{hpA4}, \eqref{hpB1}, \eqref{hpB3} on the operators $A(\cdot)$ and $B(\cdot)$, the mean value theorem and the weighted Young's inequality, we obtain that
\begin{equation}\label{i'}
\begin{split}
\int_{\Omega'} &\left\vert B\left( \left\vert \nabla u\right\vert \right) - B\left( \left\vert \nabla v\right\vert \right) \right\vert w \, dx \leq \int_{\Omega'}\left\vert B' (\xi ) \right\vert \vert \nabla w \vert w \, dx \\
\leq & \hat{C} (1+|M_A|)\int_{\Omega'}  \xi A(\xi)  \vert \nabla w \vert  w  \, dx \\
\leq & \hat{C} \bar{C} (1+|M_A|)\int_{\Omega'}  (|\nabla u| + |\nabla v|) A(|\nabla u| + |\nabla v|)  \vert \nabla w \vert  w  \, dx \\
\leq &  \hat{C} \bar{C} (1+|M_A|) \int_{\Omega'}  \sqrt{A(|\nabla u|+ |\nabla v|)} 	\ \vert \nabla w \vert \cdot (|\nabla u| + |\nabla v|) \sqrt{A(|\nabla u|+ |\nabla v|)} w  \, dx\\
\leq&   \vartheta \int_{\Omega'} A(\left\vert \nabla u\right\vert +\left\vert \nabla v\right\vert) \vert \nabla w \vert^2 \, dx \\
&+ \frac{\hat{C}^2 \bar{C}^2 (1+|M_A|)^2}{4\vartheta} \int_{\Omega'} (|\nabla u|+ |\nabla v|)^2A(\left\vert \nabla u\right\vert +\left\vert \nabla v\right\vert) w^2 \,dx \\
\leq& \vartheta \int_{\Omega'}A(\left\vert \nabla u\right\vert +\left\vert \nabla v\right\vert) |\nabla w|^2 \, dx + C(M_A, \|\nabla u\|_{L^\infty(\Omega')}, \vartheta) \int_{\Omega'} w^2 \, dx,
\end{split}
\end{equation}
where $\xi$ is an intermediate point between $|\nabla u|$ and $|\nabla v|$ and in the last line of \eqref{i'} we used the fact that the function $tA(t)$ is locally bounded. Hence, by \eqref{i} and \eqref{i'}, we have
\begin{equation}\label{ii} 
\begin{split}
C_1\int_{\Omega'}A(\left\vert \nabla u\right\vert
+\left\vert \nabla v\right\vert ) |\nabla w|^2 \, dx \leq& \vartheta \int_{\Omega'} A(\left\vert \nabla u\right\vert +\left\vert \nabla v\right\vert ) |\nabla w|^2 dx\\
&+\left[L_f + C(M_A, \|\nabla u\|_{L^\infty(\Omega')}, \vartheta)\right] \int_{\Omega'} w^2 \, dx.
\end{split}
\end{equation}
We fix $\vartheta>0$ sufficiently small, such that
$$\vartheta < C_1.$$
Hence, by \eqref{ii} it follows that
\begin{equation} 
\int_{\Omega'} A(\left\vert \nabla u\right\vert +\left\vert \nabla v\right\vert ) |\nabla w|^2 \, dx \leq \frac{L_f + C(M_A, \|\nabla u\|_{L^\infty(\Omega')}, \vartheta)}{C_1-\vartheta} \int_{\Omega'} w^2 \, dx.
\end{equation}
Using the above inequality, by Theorem \ref{WPI}, we obtain
\begin{equation}
\begin{split}
\int_{\Omega'} A(|\nabla u| + |\nabla v|) \vert \nabla w \vert^2 \, dx \leq \frac{L_f + C(M_A, \|\nabla u\|_{L^\infty(\Omega')}, \vartheta)}{C_1-\vartheta} C_{P}\left( \Omega ^{\prime }\right) \int_{\Omega'} A(|\nabla u| ) \vert \nabla w \vert^2 \, dx,
\end{split}
\end{equation}
so that, using \eqref{hpA4}, it follows
\begin{equation}
\left[1-\frac{L_f + C(M_A, \|\nabla u\|_{L^\infty(\Omega')}, \vartheta)}{C_1-\vartheta} \bar C C_{P}\left( \Omega' \right) \right] \int_{\Omega'} A(|\nabla u|+|\nabla v|) |\nabla w|^2 \, dx \leq 0.
\end{equation}
Now we note that there exists $\delta=\delta(\vartheta, f, M_A)>0$ small enough such that $|\Omega'| \leq \delta$ and
\begin{equation}
\frac{L_f + C(M_A, \|\nabla u\|_{L^\infty(\Omega')}, \vartheta)}{C_1-\vartheta} \ \bar{C} C_{P}\left( \Omega' \right) <1.
\end{equation}
Then $\left( u-v\right)^+=0$ in $\Omega'$, i.e. $u\leq v$ in $\Omega'$.

\end{proof}

\section{Symmetry and monotonicity result}\label{sec3}

In this section we prove our symmetry (and monotonicity) result. Actually we provide the details needed for the application of the \emph{moving plane method}. For the semilinear case  see \cite{BN,GNN}, in the quasilinear setting we use an adaptation of the technique developed in \cite{Da2,DP,DS1}. \\

\noindent We start with some notation: for a real number $\lambda$ we set
\begin{equation}\label{eq:sn2}
\Omega_\lambda=\{x\in \Omega:x_1 <\lambda\}
\end{equation}
\begin{equation}\label{eq:sn3}
x_\lambda= R_\lambda(x)=(2\lambda-x_1,x_2,\ldots,x_n)
\end{equation}
which is the reflection through the hyperplane $T_\lambda :=\{x\in
\mathbb R^n :  x_1= \lambda\}$. Also let
\begin{equation}\label{eq:sn4}
a=\inf _{x\in\Omega}x_1.
\end{equation}
We set
\begin{equation}\label{eq:sn33}
u_\lambda(x)=u(x_\lambda)\,.
\end{equation}
We recall that problem \eqref{equ2} is invariant up to isometries and hence we have that $u_\lambda$ satisfies
\begin{equation}  \label{wequlambda}
\int_{\Omega_\lambda} A(|\nabla u_\lambda|)\left(\nabla u_\lambda ,\nabla \psi \right) \, dx +\int_{\Omega_\lambda} B\left( |\nabla u_\lambda|\right) \psi \,dx=\int_{\Omega_\lambda} f(u_\lambda) \psi
\,dx
\end{equation}
for every $\psi \in C_{c}^{\infty }(\Omega_\lambda)$.

Finally we define
\begin{equation}\nonumber
\Lambda_0=\{a<\lambda<0 : u\leq
u_{t}\,\,\,\text{in}\,\,\,\Omega_t\,\,\,\text{for all
	$t\in(a,\lambda]$}\}\,.
\end{equation}
In the following the critical set $\mathcal{Z}_u$ will play a crucial role. Now we prove a very useful tool that we will use in the proof of our symmetry result.

\begin{lem}\label{conncomp}
	Let $u \in C^{1,\alpha}(\overline{\Omega})$ be a solution to \eqref{equ2} and $a < \lambda < 0$. If $\mathcal{C}_\lambda \subset \Omega_\lambda \setminus \mathcal{Z}_u$ is a connected component of $\Omega_\lambda \setminus \mathcal{Z}_u$ and $u \equiv u_\lambda$ in $\mathcal{C}_\lambda$, then
	$$\mathcal{C}_\lambda= \emptyset.$$
\end{lem}

\begin{proof}
	Let
	$$\mathcal{C}:=\mathcal{C}_\lambda \cup R_\lambda (\mathcal{C}_\lambda).$$
	Arguing by contradiction we assume $\mathcal{C} \neq \emptyset$.

	For all $\varepsilon>0$, let us define  $G_\varepsilon:\mathbb{R}^+_0\rightarrow\mathbb{R}$ by setting
	\begin{equation}\label{eq:G}
	G_\varepsilon(t)=\begin{cases} 0 & \text{if $0\leq t\leq \varepsilon$}  \\
	2t-2\varepsilon& \text{if $\varepsilon\leq t\leq2\varepsilon$}
	\\ t & \text{if  $t\geq 2\varepsilon$}.
	\end{cases}
	\end{equation}
	\noindent Let $\chi_{\mathcal A}$ be the characteristic function of a set $\mathcal A$.
	We define
	\begin{equation}\label{eq:concettinaanalitica}
	\Psi_{\varepsilon}\,:=\,e^{-k u}\frac{G_\varepsilon(|\nabla u|)}{|\nabla u|}\chi_{\mathcal{C}},
	\end{equation}
	where $k$ is any positive number to be chosen later. 
	
	We point out that supp$(\Psi_\varepsilon) \subset \mathcal {C}$, which implies    $\Psi_\varepsilon \in W^{1,p}_0( \mathcal {C})$. Indeed by definition of $\mathcal C$ we have that $\nabla u=0$ on $\partial \mathcal{C}$. Moreover using the test function $\Psi_{\varepsilon}$ defined in \eqref{eq:concettinaanalitica}, we are able to integrate on the boundary $\partial \mathcal {C}$ which could be not regular. 
	
	Hence, we obtain
	\begin{equation}\label{qqqqqqbissete}
	\int_{\mathcal C}\, A\left(|\nabla u|\right) \left(\nabla u ,\nabla\Psi_{\varepsilon}\right) \, dx + \int_{\mathcal C} B\left(|\nabla u|\right) \Psi_{\varepsilon} \, dx=\int_{\mathcal C} f(u)\Psi_{\varepsilon} dx.
	\end{equation}
	Since $\overline{\mathcal C}\subset \Omega$ and $f$ is a locally Lipschitz continuous positive function we have that there exists $\gamma>0$ such that 
	\[f(u)\geq \gamma.\]
	Hence
	\begin{equation}\label{qqqqqqbissete2}
	0<\gamma \int_{\mathcal C} \Psi_{\varepsilon} dx \leq  \int_{\mathcal C} f(u) \Psi_{\varepsilon} dx = \int_{\mathcal C}\, A\left(|\nabla u|\right) \left(\nabla u ,\nabla\Psi_{\varepsilon}\right) \, dx + \int_{\mathcal C} B\left(|\nabla u|\right) \Psi_{\varepsilon} \, dx.
	\end{equation}
	We set ${\displaystyle h_\varepsilon (t)=\frac{G_\varepsilon(t)}{t}}$, meaning that $h_\varepsilon(t)=0$ for $0\leq t\leq \varepsilon$. We observe that
	\begin{equation}
	\begin{split}
	\int_{\mathcal C}\, A\left(|\nabla u|\right) \left(\nabla u ,\nabla\Psi_{\varepsilon}\right) \, dx =& \int_{\mathcal C}e^{-k u}A(|\nabla u|) \left(\nabla u, \nabla \frac{G_\varepsilon(|\nabla u|)}{|\nabla u|}\right) \, dx \\
	&-k \int_{\mathcal C}e^{-k u}A(|\nabla u|) |\nabla u|^2 \frac{G_\varepsilon(|\nabla u|)}{|\nabla u|} \, dx.
	\end{split}
	\end{equation}
	
	We obtain that the first term of the right hand side of \eqref{qqqqqqbissete2} is estimated by the following:
	\begin{equation}\label{eq:smm3}
	\begin{split}
	\left| \int_{\mathcal C}e^{-k u}A(|\nabla u|) \left(\nabla u, \nabla \frac{G_\varepsilon(|\nabla u|)}{|\nabla u|}\right) \, dx\right| &\leq \int_{\mathcal C}A(|\nabla u|)|\nabla u| \ |h_\varepsilon'(|\nabla u|)||\nabla (|\nabla u|)|dx\\
	&\leq C \int_{\mathcal C}A(|\nabla u|) \Big(|\nabla u| h_\varepsilon'(|\nabla u|)\Big)\|D^2 u\| dx,
	\end{split}
	\end{equation}
	where $\|D^2 u\|$ denotes the Hessian  norm and $C$ a positive constant.
	
	\
	
	\noindent We let $\varepsilon\rightarrow 0$.
	To this aim, let us first show  that
	\begin{itemize}
		\item [$(i)$] $A(|\nabla u|) \|D^2 u\| \in L^1(\mathcal C);$
		
		\
		
		\item [$(ii)$] $|\nabla u|h_\varepsilon'(|\nabla u|)\rightarrow 0$ a.e. in $\mathcal C$ as $\varepsilon \rightarrow 0$ and $|\nabla u|h_\varepsilon'(|\nabla u|)\leq C$ with $C$ not depending on $\varepsilon$.
	\end{itemize}
	Let us  prove $(i)$. First of all we note that there exists $0 \leq \beta < 1$ such that $t^\beta A(t)$ is locally bounded. In particular, since $A(\cdot) \in \mathcal{C}^1(\R^+)$ and it holds \eqref{hpA1}, there exist $h:=h(u, \mathcal{C})>0$ and $0 \leq  \beta(\tilde \vartheta) <1$ such that $t^\beta A(t) \leq C$ for all $t \in [0,h]$ and for some $C>0$. Hence, using this information and by H\"older's inequality it follows
	\begin{equation}\label{eq:smm4}
	\begin{split}
	&\int_{\mathcal C} A(|\nabla u|) \|D^2u\| dx\leq \sqrt{|\mathcal C|}\left( \int_{\mathcal C} \left(A(|\nabla u|)\right)^2 \|D^2u\|^2 dx \right)^{\frac 12}\\
	\leq&C \left( \int_{\mathcal C}\frac{A(|\nabla u|)}{|u_i|^\beta}\|D^2u\|^2A(|\nabla u|) |u_i|^\beta dx \right )^{\frac 12}\\ 
	\leq& C \left\||\nabla u|^\beta A(|\nabla u|)\right\|_{L^{\infty}(\Omega)}^{\frac 12} \left(\int_{\mathcal C} \frac{A(|\nabla u|)}{|u_i|^\beta}\|D^2 u\|^2 dx \right)^{\frac 12},
	\end{split}
	\end{equation}
	where $C$ a positive constant. 
	
	\noindent Using Theorem \ref{theoreg1}, we infer that
	$$\left(\int_{\mathcal C} \frac{A(|\nabla u|)}{|u_i|^\beta}\|D^2 u\|^2 dx \right)^{\frac 12}\leq \mathcal{C}.$$ 
	Then, by \eqref{eq:smm4} we obtain
	$$\int_{\mathcal C} \frac{A(|\nabla u|)}{|u_i|^\beta}\|D^2 u\|^2 dx \leq \mathcal{\tilde C}.$$
	
	\
	
	\noindent Let us prove $(ii)$. Recalling \eqref{eq:G}, we obtain
	$$
	h'_\varepsilon(t)=
	\begin{cases} 0 & \text{if $0< t \leq \varepsilon$}  \\
	\frac{2\varepsilon}{t^2}& \text{if $\varepsilon< t<2\varepsilon$}
	\\ 0 & \text{if  $t\geq  2\varepsilon$} ,
	\end{cases}
	$$
	and, then, $|\nabla u|h_\varepsilon'(|\nabla u|)$ tends to $0$
	almost everywhere in $\mathcal C$ as $\varepsilon$ goes to $0$ and
	$|\nabla u|h_\varepsilon'(|\nabla u|)\leq 2$. Hence, by \eqref{eq:smm3}, we obtain that  \eqref{qqqqqqbissete2} becomes
	\begin{equation}\label{semifinal}
	\begin{split}
	0<\gamma \int_{\mathcal C} \Psi_{\varepsilon} dx \leq&  C \int_{\mathcal C}A(|\nabla u|) \Big(|\nabla u| h_\varepsilon'(|\nabla u|)\Big)\|D^2 u\| dx \\
	&-k \int_{\mathcal C}e^{-k u}A(|\nabla u|) |\nabla u|^2 h_\varepsilon(|\nabla u|) \, dx + \int_{\mathcal C}e^{-k u} B\left(|\nabla u|\right) h_\varepsilon(|\nabla u|)  \, dx.\\
	=&C \int_{\mathcal C}A(|\nabla u|) \Big(|\nabla u| h_\varepsilon'(|\nabla u|)\Big)\|D^2 u\| dx \\
	&+\int_{\mathcal C} \left(B\left(|\nabla u|\right) -k |\nabla u|^2 A(|\nabla u|)  \right) e^{-k u}h_\varepsilon(|\nabla u|) \, dx 
	\end{split}
	\end{equation}
	By \eqref{hpB1}  and \eqref{hpB3}, it is possible to deduce that there exists $h:=h(u,\Omega)>0$
	\begin{equation}\label{hpBII}
	B(t)=\int_0^tB'(s)\, ds \leq \int_0^t B'(t)\, ds \leq \hat{C} \cdot tA(t) [s]_0^t=\hat{C} \cdot t^2A(t),
    \end{equation}
	for all $t \in (0,h]$.
	
	Now, taking $h:=\|\nabla u\|_{L^\infty(\Omega)}$ and $k:=\hat{C}$, by \eqref{hpBII} we have that
	$$B\left(|\nabla u|\right) -k |\nabla u|^2 A(|\nabla u|)  \leq 0 \quad \text{in } \mathcal{C}\subset \Omega,$$
	hence \eqref{semifinal} becomes
	\begin{equation}\label{final}
	0<\gamma \int_{\mathcal C} \Psi_{\varepsilon} dx \leq  C \int_{\mathcal C}A(|\nabla u|) \Big(|\nabla u| h_\varepsilon'(|\nabla u|)\Big)\|D^2 u\| dx \\
	\end{equation}
	Finally, by the Lebesgue's dominate convergence theorem,   passing to the limit for $\varepsilon \rightarrow 0$ in \eqref{final} we obtain
	\[
	0 \geq \gamma \int_{\mathcal C} e^{-ku} dx >0.
	\]
	This gives a contradiction, hence $\mathcal{C}=\emptyset$. 
%
\end{proof}


Now having in force Lemma \ref{conncomp}, we are ready to prove our symmetry result.

\begin{proof}[Proof of Theorem \ref{symmetrybdd}]

	The proof follows via the \emph{moving plane technique}. As remarked above we recall that problem \eqref{equ2} is invariant up to isometries and hence we have that $u_\lambda$ weakly satisfies problem \eqref{equ2} according to \eqref{wequlambda}. We start showing that:
	$$\Lambda_0 \neq \emptyset\,.$$
	To prove this, let us consider $\lambda>a$ with $\lambda-a$ small enough. Let us note that, by the Hopf's boundary lemma (see \cite[Theorem 5.5.1]{PSB}) applied to a solution $u$ of problem \eqref{equ2}, we know that
	\[
	\mathcal{Z}_u\subset\subset\Omega\,.
	\]
	In particular, by this fact we deduce that $\Omega_\lambda \setminus \mathcal{Z}_u$ is connected and  that
	$$\frac{\partial u}{\partial x_1} > 0 \quad \text{in} \quad
	\Omega_\lambda \cup R_\lambda(\Omega_\lambda),$$ and this immediately proves that $u <
	u_\lambda$ in $\Omega_\lambda$.
	
	Now we can define $$\lambda_0 := \sup \Lambda_0.$$
	We shall  show that $u \leq u_\lambda$ in $\Omega_\lambda$
	for every $\lambda \in (a,0]$, namely that:
	$$\lambda_0 = 0\,.$$

	To do this we assume that $\lambda_0 < 0$ and we reach a contradiction by proving
	that $u \leq u_{\lambda_0 + \tau}$ in $\Omega_{\lambda_0 + \tau}$ for any $0 < \tau <
	\bar{\tau}$ for some small $\bar{\tau}>0$. 
	
	As a consequence of Corollary \ref{theoreg2} we remark that $|\mathcal{Z}_{u}|=0$, see also  \cite{DamSciunzi}. Let $\mathcal{A}\subset \Omega_{\lambda_0}$ be an open set such that $\mathcal Z_u\cap\Omega_{\lambda_0}\subset \mathcal{A} \subset \subset \Omega$. Such set exists by the Hopf's boundary Lemma. Moreover note  that, since $|\mathcal Z_u|=0$, we can take $\mathcal{A}$ of arbitrarily small measure. By 	continuity we know that $u \leq u_{\lambda_0}$ in $\Omega_{\lambda_0}$.
	We can exploit the strong comparison principle, see e.g. \cite[Theorem 2.5.2]{PSB}  to get that, in any connected component of $\Omega_{\lambda_{0}}\setminus \mathcal Z_u$, we have
	$$u<u_{\lambda_0} \qquad\text{or}\qquad u\equiv u_{\lambda_0}.$$
	The case $u\equiv u_{\lambda_0}$ in some  connected component
	$\mathcal{C}$ of $\Omega_{\lambda_{0}}\setminus \mathcal Z_u$ is not
	possible, since by symmetry, it would imply the existence of  a local symmetry phenomenon and consequently that $\Omega_{\lambda_0} \setminus \mathcal  Z_u$ would be not connected,  in  spite of what we proved in Lemma \ref{conncomp}. Hence we deduce that $u <
	u_{\lambda_0}$ in $\Omega_{\lambda_0}$. Therefore, given a compact set $\mathcal{K} \subset
	\Omega_{\lambda_0} \setminus \mathcal{A}$, by
	uniform continuity we can ensure that $u < u_{\lambda_0+\tau}$ in
	$\mathcal{K}$ for any $0 < \tau < \bar{\tau}$ for some small $\bar{\tau}>0$.
    Now we choose $\bar \tau>0$ sufficiently small and $|\mathcal{K}|$ big enough such that $|\Omega_{\lambda_0+\tau} \setminus \mathcal{K}| < \delta$, where $\delta>0$ is given by Theorem \ref{WCP} and for any $0<\tau < \bar \tau$. Moreover, noticing that $u \leq u_{\lambda_0 + \tau}$ on $\partial \left(\Omega_{\lambda_0+\tau} \setminus \mathcal{K}\right)$ , by Theorem \ref{WCP} it follows that
    $$u \leq u_{\lambda_0 + \tau} \qquad \text{in} \; \Omega_{\lambda_0+\tau} \setminus \mathcal{K},$$
    for any $0<\tau < \bar \tau$ and taking into account that $u < u_{\lambda_0+\tau}$ in $\mathcal{K}$, we have
    $$u \leq u_{\lambda_0 + \tau} \qquad \text{in} \; \Omega_{\lambda_0+\tau},$$
    for any $0<\tau < \bar \tau$. This fact gives a contradiction with the definition of $\lambda_0$, and so we have
    $$\lambda_0=0.$$
    Since the moving plane procedure can be performed in the same way
	but in the opposite direction, then this proves the desired symmetry
	result. The fact that the solution is increasing in the
	$x_1$-direction in $\{x_1 < 0\}$ is implicit in the moving plane
	procedure.

	Finally, if $\Omega$ is a ball, repeating this argument along any direction, it follows that $u$ is radially symmetric. The fact that
	$\displaystyle \frac{\partial u}{\partial r}(r)<0$ for $r\neq 0$, follows by the Hopf's boundary lemma which works in this case since the level sets are balls
	and therefore fulfill the interior sphere condition.
\end{proof}

\end{document}